\documentclass[12pt,letterpaper]{article}

\pdfoutput=1

\usepackage{graphics,epsfig,graphicx,color}
\graphicspath{{../Figures/}}
\usepackage[caption = false]{subfig}
\usepackage{float}
\usepackage{capt-of}
\usepackage{diagbox}

\usepackage{amssymb,latexsym}

\usepackage{setspace}

\usepackage{amsmath}

\usepackage{mathrsfs,amsfonts}

\usepackage{amsthm,amsxtra}

\usepackage[final]{showkeys}

\usepackage{stmaryrd}

\usepackage{algorithm}
\usepackage{algorithmicx}
\usepackage{algpseudocode}

\usepackage{etaremune}
\usepackage{enumitem} 

\newif\ifPDF
\ifx\pdfoutput\undefined
	\PDFfalse
\else
	\ifnum\pdfoutput > 0
		\PDFtrue
	\else
		\PDFfalse
	\fi
\fi

\ifPDF
	\usepackage{pdftricks}
	\begin{psinputs}
		\usepackage{pstricks}
		\usepackage{pstcol}
		\usepackage{pst-plot}
		\usepackage{pst-tree}
		\usepackage{pst-eps}
		\usepackage{multido}
		\usepackage{pst-node}
		\usepackage{pst-eps}
	\end{psinputs}
\else
	\usepackage{pstricks}
\fi

\ifPDF
	\usepackage[debug,pdftex,colorlinks=true, 
	linkcolor=blue, bookmarksopen=false,
	plainpages=false,pdfpagelabels]{hyperref}
\else
	\usepackage[dvips]{hyperref}
\fi

\usepackage{fancyhdr}

\usepackage{comment}

\usepackage[toc,page]{appendix}

\usepackage{cancel}

\usepackage{multirow}
\usepackage{tabularx}

\usepackage[openbib]{currvita}
\usepackage{bibentry}
\usepackage{cleveref}

\usepackage{mathtools}
\usepackage{extarrows}

\usepackage[titles,subfigure]{tocloft}

\newtheorem{theorem}{Theorem}[section]
\newtheorem{lemma}[theorem]{Lemma}

\newtheorem{remark}[theorem]{Remark} 
\newtheorem{corollary}[theorem]{Corollary}

\newcommand{\dint}{\displaystyle\int}

\newcommand{\eps}{\varepsilon}

 \newcommand{\bbE}{\mathbb E}

\newcommand{\bbR}{\mathbb R} \newcommand{\bbS}{\mathbb S}

\newcommand{\be}{\mathbf e}

\newcommand{\bk}{\mathbf k} 
 \newcommand{\bn}{\mathbf n}
 \newcommand{\bp}{\mathbf p}
\newcommand{\bq}{\mathbf q} 
 
 \newcommand{\bv}{\mathbf v} 
 \newcommand{\bx}{\mathbf x} 
\newcommand{\by}{\mathbf y} \newcommand{\bz}{\mathbf z}

 \newcommand{\cD}{\mathcal D}

\newcommand{\cI}{\mathcal I} 
\newcommand{\cK}{\mathcal K} \newcommand{\cL}{\mathcal L}
\newcommand{\cM}{\mathcal M} 
\newcommand{\cO}{\mathcal O}

\newcommand{\cW}{\mathcal W}

\newcommand{\aver}[1]{\langle {#1} \rangle}

\newcommand{\wt}{\widetilde}

\newenvironment{keywords}
{\noindent{\bf Key words.}\small}{\par\vspace{1ex}}

\newcommand{\DELETE}[1]{}

 \newcommand{\YZ}[1]{#1}
	
\title{Transport models for wave propagation in scattering media with nonlinear absorption}
\author{
Joseph Kraisler\thanks{Department of Applied Physics and Applied Mathematics, Columbia University, New York, NY, 10027; \href{mailto:jek2199@columbia.edu}{jek2199@columbia.edu}}
\and
Wei Li\thanks{Department of Mathematical Sciences, Depaul University, Chicago, IL 60604; \href{mailto:wei.li@depaul.edu}{wei.li@depaul.edu}}
\and
Kui Ren\thanks{Department of Applied Physics and Applied Mathematics, Columbia University, New York, NY 10027;\href{mailto:kr2002@columbia.edu}{kr2002@columbia.edu}}
\and
John C. Schotland\thanks{Department of Mathematics and Department of Physics, Yale University, New Haven, CT 06511; \href{mailto:john.schotland@yale.edu}{john.schotland@yale.edu}}
\and
Yimin Zhong\thanks{Department of Mathematics and Statistics, Auburn University, Auburn, AL 36830; \href{mailto:yimin.zhong@auburn.edu}{yimin.zhong@auburn.edu}}
}
\date{}
\begin{document}

\maketitle

\begin{abstract}
This work considers the propagation of high-frequency waves in highly-scattering media where physical absorption of a nonlinear nature occurs. Using the classical tools of the Wigner transform and multiscale analysis, we derive semilinear radiative transport models for the phase-space intensity  and the diffusive limits of such transport models. As an application, we consider an inverse problem for the semilinear transport equation, where we reconstruct the absorption coefficients of the equation from a functional of its solution. We obtain a uniqueness result on the inverse problem.
\end{abstract}

\section{Introduction}


The derivation of kinetic models for wave propagation in highly-scattering media is a classical subject~\cite{Ishimaru-Book97,CaSc-Book21} that has received significant attention in the past two decades due to its importance in many emerging applications~\cite{BaSaCoClGu-WM14,BoGaSo-MMS19,GaSo-SIAM08,Margerin-SE05,MiTrLa-Book06,KrSc-JMP22}. A significant amount of progress has been made on both the mathematical justification of the derivation (such as those based on multiscale analysis of the Wigner transform)~\cite{BaPaRy-Nonlinearity02,Erdos-Yau2,Fannjiang-CRP07,GuWa-JMP99,LuSp-ARMA07,RyPaKe-WM96} and computational validation of the derived kinetic models~\cite{BaPi-WM06,BaRe-SIAM08,JiXi-JOSA10,RoMaVuThGr-JOSA01,VoScKi-OL09}; see ~\cite{AkSa-MMS21,AmBoGaJiSe-JMP13,Bal-WM05,BaKoRy-KRM10,BoGa-PRE16,FaRy-SIAM01,GaSo-SIAM08,Gomez-JMPA12} and references therein for additional investigations in this field. The obtained models for imaging in complex media have also been utilized in many different settings~\cite{BaPi-SIAM07,BaRe-SIAM08,BoGaSo-MMS19,ChLi-RMS21}. 

In this work, we are interested in developing kinetic models when nonlinear absorption occurs during wave propagation~\cite{BaReZh-JBO18,Mahr-QE12,RuPe-AOP10,YiLiAl-AO99}. We are mainly motivated by applications where reconstructing the absorption of the underlying medium from internal or boundary observations is of practical interest. Such applications include, for instance, the case of reconstructing the two-photon absorption coefficient of biological tissues with optical or photoacoustic measurements~\cite{DeStWe-Science90,LaReZh-SIAM22,MoViDi-APL08,ReZh-SIAM18,ReZh-SIAM21,StZh-SIAM22,XuWaXuZhYi-OC21,YaNaTa-OE11,YeKo-OE10}.

Our derivation will be carried out in the frequency domain, where wave propagation is described by the standard Helmholtz equation. The nonlinear absorption mechanism we are interested \YZ{in} is modeled as a zeroth-order perturbation to the second-order Helmholtz differential operator. This is the essential factor that makes it possible for us to perform the standard calculations in this field for our nonlinear problem. Even though this is a mathematically less attractive nonlinearity to study, the derivation does provide us a formal justification of the semilinear radiative transport model, see, for instance~\eqref{EQ: TRANS}, used in many applications. We refer interested readers to~\cite{FaJiPa-SIAM03} for the derivation of transport models for a different type of nonlinearity that makes use of a mean-field approximation. \YZ{Furthermore, we emphasize that the derived nonlinear transport model has a variety of important applications. In fluorescence imaging, it can be used to localize and characterize fluorescent molecules by determining their two-photon or multi-photon absorption properties. }

The remainder of this paper is organized as follows. In Section~\ref{SEC:ERT}, we derive the radiative transport model for media with quadratic and higher-order absorption. We then discuss the diffusive limit of the derived transport models in Section~\ref{SEC:Diff}. As an application of the derived model, we study in Section~\ref{SEC:IP} an inverse medium problem for our semilinear radiative transport equation. Concluding remarks are offered in Section~\ref{SEC:Concl}.

\section{Derivation of the transport equation}
\label{SEC:ERT}

For simplicity, we first consider the case of quadratic nonlinear absorption. This could serve as a model of light propagation in media with two-photon absorption~\cite{StZh-SIAM22,XuWaXuZhYi-OC21}. We will then generalize the result to the case of a general polynomial nonlinearity. Let the wave field $p(z, \bx)$ be the solution to the scalar wave equation in the time-harmonic form, that is,
\begin{equation}\label{EQ: WAVE}
\Delta_{\bx} p + \frac{\partial^2 p}{\partial z^2} + k^2n^2 p = 0,
\end{equation}
where $\Delta_{\bx}$ is the transverse Laplacian in $\bx\in\bbR^{d}$ ($d\ge 1$), $k$ is the wave number, and $n = n(z, \bx, p)$ is the refractive index. We assume that the refractive index takes the form
\begin{equation}
    n^2 = 1 - 2\sigma V(\frac{z}{\ell_z}, \frac{\bx}{\ell_{\bx}}) + i k^{-1}\mu \widetilde K(z, \bx) |p|^2\,,
\end{equation}
where $V$ is a real bounded stationary random field with zero mean, with $\ell_{\bx}$ and $\ell_z$ are the transverse and longitudinal correlation lengths of the random field, respectively. The deterministic function $\widetilde K$ is non-negative and measures the strength of the second-order absorption.  The parameters $\sigma$ and $\mu$ are the scaling factors quantifying the amplitudes of the fluctuation and the second-order absorption, respectively.
Assuming that the field $p$ possesses a beam-like structure propagating in $z$ direction, we may write $p(z, \bx) = e^{ik z} \psi(z, \bx)$ with complex amplitude $\psi(z, \bx)$ satisfing the following equation
\begin{equation}\label{EQ: BEAM}
    \frac{\partial^2 \psi}{\partial z^2} + 2 ik \frac{\partial \psi}{\partial z} + \Delta_{\bx}\psi + k^2 (n^2 - 1) \psi = 0\,.
\end{equation}
Let $L_{\bx}$ and $L_z$ be the characteristic lengths of propagation in $\bx$ and $z$ directions, respectively. We rescale the variables $\bx\mapsto L_{\bx} \bx$ and $z\mapsto L_z z$ and $\bx$, $z$ are now $\cO(1)$.  Then with the newly-defined variables, we may write $(n^2 - 1)$ as 
\begin{equation}
    n^2 - 1 = -2\sigma V(\frac{L_z z}{\ell_z}, \frac{L_{\bx}\bx}{\ell_{\bx}}) + i k^{-1}\mu \widetilde K(L_z \bz, L_{\bx}\bx) |\YZ{p}|^2\,.
\end{equation}
The equation~\eqref{EQ: BEAM} now becomes
\begin{equation}
        \frac{1}{L_z^2}\frac{\partial^2 \psi}{\partial z^2} + 2 \frac{ik}{L_z} \frac{\partial \psi}{\partial z} + \frac{1}{L_{\bx}^2}\Delta_{\bx}\psi - 2k^2\sigma V( \frac{L_z z}{\ell_z}, \frac{L_{\bx}\bx}{\ell_{\bx}}) \psi + i k \mu \widetilde K( L_z z, L_{\bx}\bx)|\psi|^2 \psi = 0\,.
\end{equation}
With the small aperture assumption that $L_{\bx}\ll L_z$, we can formally approximate the above equation by the paraxial wave equation 
\begin{equation}
i\frac{\partial \psi}{\partial z} + \frac{L_z}{2kL_{\bx}^2}\Delta_{\bx}\psi - k L_z \sigma V( \frac{L_z z}{\ell_z}, \frac{L_{\bx}\bx}{\ell_{\bx}}) \psi + i\frac{L_z}{2} \mu \widetilde K( L_z z, L_{\bx}\bx)|\psi|^2 \psi = 0\,.
\end{equation}
Our derivation works in the regime where the longitudinal propagation distance $L_z$ is much larger than the correlation length $\ell_z$ and the correlation length is much larger than the wavelength, that is $L_z \gg \ell_z$ and $\ell_z \gg \lambda:=\frac{2\pi}{k}$. We, therefore, introduce the small parameter $\eps$, and assume the scaling relations in the weak-coupling regime:
\begin{equation}\label{eq:epsmult}
\frac{\ell_{\bx}}{L_{\bx}} = \frac{\ell_z}{L_z} = \eps \ll 1,\quad k \ell_{\bx}^2 = \ell_z ,\quad \sigma = \frac{1}{k \ell_z}\sqrt{\eps}, \quad \mu \propto \frac{1}{L_z}.
\end{equation}
Let us denote the rescaled wave field by $\psi_{\eps}( z, \bx)$ and take $K(z, \bx) = L_z\mu\widetilde{K}( L_z z, L_{\bx}\bx)$. Then the paraxial wave equation turns into
\begin{equation}
    i\frac{\partial\psi_{\eps}}{\partial z} + \frac{\eps}{2}\Delta_{\bx}\psi_{\eps} - \frac{1}{\sqrt{\eps}} V( \frac{z}{\eps}, \frac{\bx}{\eps}) \psi_{\eps} + \frac{i}{2}K( z, \bx)|\psi_{\eps}|^2\psi_{\eps} = 0.
\end{equation}
We then take the Wigner transform of $\psi_{\eps}$:
\begin{equation}
    W_{\eps}(z, \bx, \bk) = \int_{\bbR^d} e^{i\bk\cdot \by}\psi_{\eps}(\bx - \frac{\eps\by}{2}, z) \overline{\psi_{\eps}(\bx+\frac{\eps\by}{2}, z)} \frac{d\by}{(2\pi)^d},
\end{equation}
It is then standard to check $W_{\eps}(z, \bx, \bk)$ satisfies the Liouville equation:
\begin{equation}\label{EQ: LIOU}
    \frac{\partial W_{\eps}}{\partial z} + \bk \cdot \nabla_{\bx} W_{\eps} + \cL_{V,\eps} W_{\eps} + \cL_{K, \eps} W_{\eps} = 0,
\end{equation}
where
\begin{equation}
\begin{aligned}
    \cL_{V, \eps} W_{\eps} &= \frac{i}{\sqrt{\eps}}  \int_{\bbR^d} e^{-i\bp\cdot \bx / \eps } \left(W_{\eps}(z, \bx, \bk+\frac{\bp}{2}) - W_{\eps}(z, \bx, \bk - \frac{\bp}{2})\right)  \YZ{\widehat{V}( \frac{z}{\eps}, \bp )d\bp}   , \\
    \cL_{K, \eps} W_{\eps} &= \frac{1}{2}\int_{\bbR^d} e^{-i\bp\cdot x} \left(W_{\eps}(z, \bx, \bk+\frac{\eps \bp}{2}) + W_{\eps}( z, 
\bx, \bk - \frac{\eps \bp}{2})\right)  \YZ{\widehat{S}_{\eps}(z, \bp)d\bp}\,.
\end{aligned}
\end{equation}
Here $S_{\eps}(z, \bx)$ is defined as
\begin{equation}\label{EQ: S}
\begin{aligned}
    S_{\eps}(z, \bx) := K(z, \bx) |\psi_{\eps}(z, \bx)|^2 = K(z, \bx) \int_{\bbR^d} W_{\eps}(  z, \bx, \bk) d\bk\,,
\end{aligned}
\end{equation}
while $\widehat{V}$ and $\widehat{S}_{\eps}$ denote the Fourier transform ($\bx\to \bp$) of $V$ and $S_{\eps}$ respectively. We use the standard Fourier transform definition
\begin{equation*}
	\widehat{f}(\bp)=\int_{\bbR^d} e^{i\bp\cdot \bx}f(\bx) \frac{d\bx}{(2\pi)^d}\,.
\end{equation*}

\subsection{Multiscale expansion}

In order to find the asymptotic limit as $\eps\to 0$, we introduce $\by =\bx/\eps $ as the fast variable and denote $W_{\eps}(z, \bx, \bk) = W_{\eps}(z, \bx, \by, \bk)$. Formally we write $W_{\eps}$ as an asymptotic expansion in $\eps$:
\begin{equation*}
    W_{\eps}(z, \bx, \by, \bk) = W_0(z, \bx, \by, \bk) + \sqrt{\eps} W_1(z, \bx, \by, \bk) + \eps W_2 (z, \bx, \by, \bk) + \cdots\,.
\end{equation*}
\YZ{For the linear model, the main theoretical difficulty in the rigorous derivation of the transport equation is to estimate the remainder in the above ansatz~\cite{Erdos-Yau2}. For the two-photon absorption nonlinear model, an additional difficulty is brought about by the nonlinear term $\cK_{K,\eps} W_{\eps}$, which arises at the $\cO(1)$ scale. Additional regularity of $W_0$ is required to obtain a suitable estimate.} Using~\eqref{EQ: S}, we may also expand $S_{\eps}(z, \bx) = S_{\eps}(z, \bx, \by)$ accordingly as
\begin{equation*}
    S_{\eps}(z, \bx, \by) = S_0(z, \bx, \by) + \sqrt{\eps} S_1(z, \bx, \by) + \eps S_2(z, \bx, \by) + \cdots\,.
\end{equation*}
where 
\begin{equation*}
   S_0(z, \bx, \by) = K(z, \bx) \int_{\bbR^d} W_0(z, \bx, \by, \bk)d\bk\,.
\end{equation*}
We can now plug in the transform $\nabla_{\bx} \to \nabla_{\bx}+\frac{1}{\eps}\nabla_{\by}$ into~\eqref{EQ: LIOU} to conclude that the leading order equation at $\cO(\eps^{-1})$ implies $\bk \cdot \nabla_{\by} W_0 = 0 $. This is equivalent to
\[
	W_0 (z, \bx, \by, \bk) = W_0(z, \bx, \bk), \ \ \ \mbox{and}\ \ \ S_0(z, \bx,\by) = S_0(z, \bx)\,.
\]
For the order of $\cO(\eps^{-1/2})$, we have
\begin{equation}
    \bk\cdot \nabla_{\by} W_1 + \alpha W_1 + i \int_{\bbR^d} e^{-i\bp\cdot \by} (W_0(z, \bx, \bk+\frac{\bp}{2}) - W_0(z, \bx, \bk-\frac{\bp}{2})) \widehat{V}( \frac{z}{\eps}, \bp) \YZ{d\bp} = 0\,,
\end{equation}
where $\alpha\to 0^{+}$. This gives that the Fourier transform of $W_1$, $\widehat{W}_1$,  is
\begin{equation}
    \widehat{W}_1(z, \bx, \bp, \bk) = \frac{ (W_0(z, \bx, \bk+\frac{\bp}{2}) - W_0(z, \bx, \bk-\frac{\bp}{2})) \widehat{V}(\frac{z}{\eps}, \bp) }{\bp \cdot \bk + i\alpha}\,.
\end{equation}
Finally, we derive the equation for $\cO(1)$ terms. \YZ{In order to handle the nonlinearity, we impose some regularity assumptions.} Let $s > d+2$ and assume there exist positive constants $C_1$, $C_2$ and $C_3$ such that
\begin{equation*}
\begin{aligned}
    \|W_0(z, \bx, \cdot) \|_{C^1(\bbR^d)} + \|W_0(z, \bx, \cdot) \|_{L^1(\bbR^d)} < C_1,\quad &\forall (z, \bx)\in \bbR^{d+1}, \\
   \left\|\int_{\bbR^d} W_0(z, \cdot, \bk) d\bk  \right\|_{H^s(\bbR^{d})} < C_2,\quad &\forall z \in\bbR, \\
   \|K(z, \cdot)\|_{H^s(\bbR^d)} < C_3 ,\quad& \forall z\in\bbR\,.
\end{aligned}
\end{equation*}
 Then we have that
\begin{equation}
    W_{0}(z, \bx, \bk - \frac{\eps}{2}\bp
     ) + W_{0}(z, \bx, \bk + \frac{\eps}{2}\bp
     ) =  2W_{0}(z, \bx, \bk) + R(z, \bx, \bk)\,,
\end{equation}
where \YZ{$|R(z,\bx, \bk)|\le C_1 \eps|\bp|$}   . Moreover, we have that
\begin{equation}
\begin{aligned}
     \left|\int_{\bbR^d}   \frac{
   |\bp| |\widehat{S}_0(z, \bp)|d\bp}{
  (2\pi)^d} \right|^2 &\le \frac{1}{(2\pi)^{2d}} \int_{\bbR^d} \frac{|\bp|^2}{(1 + |\bp|^2)^{s/2}} d\bp \int_{\bbR^d} (1 + |\bp|^2)^{\frac{s}{2}} |\widehat{S}_0(z, \bp)|^2 d\bp \\
  &=  \frac{1}{(2\pi)^{2d}}\left(   \int_{\bbR^d} \frac{|\bp|^2}{(1 + |\bp|^2)^{s/2}} d\bp \right)\|S_0(z, \cdot)\|_{H^s(\bbR^d)}^2 \\
  &\le  C_s \|K(z, \cdot)\|_{H^s(\bbR^d)}^2 \left\|\int_{\bbR^d} W_0(z, \cdot, \bk) d\bk\right\|_{H^s(\bbR^d)}^2
\end{aligned}
\end{equation}
is also uniformly bounded, where $C_s > 0$ is a constant depending only on $s$. The last inequality holds since $s>d+2$ implies $s>d/2$. Therefore the $\cO(1)$ term is
\begin{equation}
\begin{aligned}
    &\frac{\partial W_0}{\partial z} + \bk\cdot \nabla_{\bx} W_0 + \bk \cdot \nabla_{\by} W_2  \\
    &+ i \int_{\bbR^d} e^{-i\bp\cdot \by} (W_1 (z, \bx, \by, \bk+\frac{\bp}{2}) - W_1 (z, \bx, \by, \bk-\frac{\bp}{2}) ) \widehat{V}( \frac{z}{\eps}, \bp) \YZ{d\bp}\\
    & + W_0(z, \bx, \bk) S_0(z, \bx)= 0\,.
\end{aligned}
\end{equation}
In order to close the equation, we still need to add the orthogonal relation between $W_0$ and $W_2$, that is, $\bbE[\bk\cdot \nabla_{\by} W_2] = 0$. Hence, we have
\begin{equation}\label{EQ: EXP}
\begin{aligned}
    &\frac{\partial W_0}{\partial z} + \bk\cdot \nabla_{\bx} W_0 \\
    &+\bbE\left [ i \int_{\bbR^d} e^{-i\bp\cdot \by} (W_1 (z, \bx, \by, \bk+\frac{\bp}{2}) - W_1 (z, \bx, \by, \bk-\frac{\bp}{2}) ) \widehat{V}( \frac{z}{\eps}, \bp) \YZ{d\bp}\right] \\
    &+ W_0(z, \bx, \bk) S_0(z, \bx) = 0\,.
\end{aligned}
\end{equation}
Let $R$ be the correlation function of $V$, and assume that the power spectrum satisfies
\begin{equation}
    \bbE[\widehat{V}(z, \bp)\widehat{V}(z, \bq)] = \YZ{ \widehat{R}(\bp)\delta(\bp+\bq) } \,.
\end{equation}
Then \YZ{as $\eps\to 0^{+}$}, the expectation term in~\eqref{EQ: EXP} converges weakly to
\begin{equation*}
\begin{aligned}
    &\bbE \left[i\int_{\bbR^d} e^{ -i\bp\cdot \bx/\eps} (W_1(z,  \bx, \frac{\bx}{\eps},  \bk+\frac{\bp}{2} ) - W_1(z, \bx, \frac{\bx}{\eps}, \bk - \frac{\bp}{2})) \widehat{V}(\frac{z}{\eps}, \bp) \YZ{d\bp}\right] \\
   & \to 4\pi \int_{\bbR^d} \widehat{R}(\bp-\bk) [W_0(z,\bx,\bk) - W_0(z,\bx,\bp)] \delta(|\bk|^2 - |\bp|^2) d\bp \,.
\end{aligned}
\end{equation*}
Therefore the final radiative transport equation of $W_0$ is 
\begin{equation}\label{EQ: TRANS}
\begin{aligned}
    &\frac{\partial W_0}{\partial z} + \bk\cdot \nabla_{\bx} W_0 + 4\pi\int_{\bbR^d} \widehat{R}(\bp-\bk) [W_0(z, \bx, \bk) - W_0(z, \bx,  \bp)] \delta(|\bk|^2 - |\bp|^2)  d\bp \\
    & + \left( K(z, \bx)\int_{\bbR^d} W_0(z, \bx,  \bk') d\bk'\right) W_0(z, \bx, \bk) = 0\,.
 \end{aligned}
\end{equation}
\YZ{Physically, the last term on the left side of ~\eqref{EQ: TRANS} accounts for quadratic absorption, which indicates that the absorption coefficient linearly depends on the angular average of $W_0(\bx, z, \bk)$. Although $W_0(z,\bx, \bk)$ is not guaranteed to be a non-negative quantity, the angular average $\int_{\bbR^d}W_{0}(z,\bx, \bk)d\bk = \lim_{\eps\to 0}|\psi_{\eps}|^2$ is always non-negative.} 

\subsection{Extension to higher order absorption}

We now extend the previous result to the case of general polynomial absorption. This could be a physical model for the propagation of light in media with multi-photon absorption~\cite{Boyd-Book20}.  The absorption term in the refractive index is modeled by 
\begin{equation}
  n^2 = 1 - 2\sigma V( \frac{z}{\ell_z}, \frac{\bx}{\ell_{\bx}}) +  i k^{-1}\mu \sum_{l=0}^{\YZ{L}} \widetilde{K}_l(z, \bx) |\YZ{p}|^{2l},
\end{equation}
where $\widetilde{K}_l$ stands for the absorption strength of $(l+1)$-th order. Following the same derivation of the paraxial wave equation, we have a new Liouville equation in the form of ~\eqref{EQ: LIOU} where the only modification is the term
\begin{equation}
    S_{\eps}(\bx, z) = \sum_{l = 0}^{L} K_l(z, \bx) |\psi_{\eps}(z, \bx)|^{2l} ,\quad K_l(z, \bx) := \mu L_z\widetilde K_l( L_z z, L_{\bx}\bx)\,.
\end{equation}
Assuming that $K_l$ and $W_0$ are sufficiently regular, we can follow the same procedure and obtain the radiative transport equation for $W_0$ in the setting of the polynomial absorption: 
\begin{equation}\label{EQ: M TRANS}
\begin{aligned}
    &\frac{\partial W_0}{\partial z} + \bk\cdot \nabla_{\bx} W_0 + 4\pi \int_{\bbR^d} \widehat{R}(\bp-\bk) [W_0(z, \bx, \bk) - W_0(z, \bx,  \bp)] \delta(|\bk|^2 - |\bp|^2)  d\bp \\
    & + \left(\sum_{l=0}^L K_l(z, \bx)\left[\int_{\bbR^d} W_0(z, \bx,  \bk') d\bk'\right]^l\right) W_0(z, \bx, \bk) = 0\,.
 \end{aligned}
\end{equation}

\subsection{The nonlinear RTE in the standard form}

For a monochromatic solution $W_0(z, \bx,\bk)$ which is supported on $\|\bk\|=1$, the above radiative transport equation becomes 
\begin{equation}\label{EQ: RTE TRANS}
\begin{aligned}
    &\frac{\partial W_0}{\partial z} + \bk\cdot \nabla_{\bx} W_0 + 4\pi \int_{\bbS^{d-1}} \widehat{R}(\bp-\bk) [W_0(z, \bx, \bk) - W_0(z, \bx,  \bp)] d\bp \\
    & + \left(\sum_{l=0}^L K_l(z, \bx)\left[\int_{\bbS^{d-1}} W_0(z, \bx,  \bk') d\bk'\right]^l\right) W_0(z, \bx, \bk) = 0\,,
 \end{aligned}
\end{equation}
where $\bbS^{d-1}$ is the unit sphere in $\bbR^d$. It can further be put in the standard form
\begin{equation}\label{EQ: RTE 0}
\begin{aligned}
    \frac{\partial W_0}{\partial z} + \bk\cdot \nabla_{\bx} W_0 + \Sigma_a(\aver{W_0}) W_0+ \Sigma_s (W_0 - \cK W_0 ) &= 0.
\end{aligned}
\end{equation}
Here the total energy of the field at $\bx$ is denoted by
$$\aver{W_0}:= \int_{\bbS^{d-1}} W_0(z, \bx, \bk)d\bk$$
and
\begin{equation}\label{EQ:Poly Absorption}
	\Sigma_a(\aver{W_0}) =\sum_{l=0}^{L} \Sigma_{a, l} \aver{W_0}^l
\end{equation}
is the effective absorption coefficient with $\Sigma_{a, l}(z, \bx) = K_{l}(z, \bx)$ is the absorption coefficient of $(l+1)$-th order. The scattering coefficient is $\Sigma_s(\bk) := 4\pi \dint_{\bbS^{d-1}} \widehat{R}(\bp - \bk) d\bp$ and the scattering operator $\cK$ is defined as
\begin{equation}
    \cK u(z, \bx, \bk) := \int_{\bbS^{d-1}} p(\bk, \bk') u(z, \bx, \bk') d\bk'
\end{equation}
where the scattering phase function $$ p(\bk, \bk') = \frac{4\pi \widehat{R}(\bk - \bk')}{\Sigma_s(\bk)}.$$ For the problem to be physical, the absorption strengths $K_l$ ($0\le l\le L$) should be all non-negative. 

\section{The diffusion limit}
\label{SEC:Diff}

We now study the diffusion limit of the transport equation for $W_0(z, \bx, \bk)$ with general polynomial nonlinear absorption. We assume that the physical domain $\Omega$ \YZ{is bounded and convex} with smooth boundary $\partial\Omega$. We focus on the following nonlinear radiative transport equation:
\begin{equation}\label{EQ: RTE 1}
\begin{aligned}
    \frac{\partial W_0}{\partial z} + \bk\cdot \nabla_{\bx} W_0 + \Sigma_a(\aver{W_0}) W_0+ \Sigma_s (W_0 - \cK W_0 ) &= 0\quad &\text{ in }&X,\\
    W_0(z, \bx, \bk) &= 0\quad&\text{ on }&\Gamma_{-},\\
    W_0(0, \bx, \bk) &= f(\bx)\quad &\text{ on }& X_0,
\end{aligned}
\end{equation}
where $X = (0, T)\times \Omega \times \bbS^{d-1}$, $\Gamma_{-} := \{(z, \bx, \bk) \in (0, T)\times \partial \Omega \times\bbS^{d-1}\mid \bn_{\bx}\cdot \bk < 0\}$, and $\bn_{\bx}$ is the unit outward normal vector at $\bx\in\partial \Omega$, $X_0 = \Omega \times \bbS^{d-1}$. 

We will focus on power spectra of the form $\widehat{R}(\bp - \bk) = \widehat{R}(\bp \cdot \bk)$, in which case the scattering coefficient $$\Sigma_s(\bk) \equiv \Sigma_s$$ is a constant, and the scattering phase function is $p(\bk, \bk') = 4\pi \widehat{R}(\bk \cdot \bk')/\Sigma_s$. 
Assume the scattering phase function $p(\bk\cdot \bk')$ is bounded below and above by positive constants $\underline{\theta}, \overline{\theta} \ge 0$ such that
\begin{equation}\label{EQ: THETA}
  \underline{\theta} \le  p(\bk, \bk') \le \overline{\theta}\,.
\end{equation}
Note that $\underline{\theta}$ satisfies the condition  $\nu_{d-1}\underline{\theta}\leq 1$, where $\nu_{d-1}$ is the measure of the ($d-1$) dimensional unit sphere. For simplicity, we require that the initial condition $f(\bx)\in L^{\infty}(\Omega\times\bbS^{d-1})$ is $\bx$-dependent and non-negative. The absorption coefficients obey the condition
\begin{equation}\label{EQ:Absorption Coeff}
   0< \underline{\Sigma_{a,l} }\le \Sigma_{a, l} (z, \bx) \le \overline{\Sigma_{a,l} } <\infty ,
\end{equation}
for some constants $\underline{\Sigma_{a,l} }$ and $\overline{\Sigma_{a,l}}$. 

We need the following lemma to have a well-posed problem.
\begin{lemma} [Kellogg~\cite{Kellogg-PAMS76}]\label{LEM: KEL}
Let $\cM$ be a bounded convex open subset of a real Banach space, and $F : \overline{
\cM} \to \overline{\cM}$ is a compact continuous map which is continuously Fr\'echet differentiable on \YZ{$\cM$}. If (i) for each $m\in \cM$, $1$ is not an eigenvalue of $F'(m)$, and (ii) for
each $m\in\partial \cM$, $m\neq F(m)$, then $F$ has a unique fixed point in $\cM$.
\end{lemma}
We define the space, using $dS(\bx)$ to denote the surface measure of $\partial\Omega$, 
\begin{equation*}
\begin{aligned}
    \cW_p = \{ u\in L^p((0, T)\times \Omega\times  \bbS^{d-1} ); \frac{\partial u}{\partial z} + \bk\cdot \nabla u \in L^p((0, T)\times  \Omega\times \bbS^{d-1} ); \\
    u(0, \cdot, \cdot)\in L^p(\Omega\times\bbS^{d-1}); u|_{\Gamma_{-}}\in L^p( (0,T)\times \Gamma_{-}, |\bn_{\bx}\cdot \bk| dz\, d\bk dS(\bx))\}
\end{aligned}
\end{equation*}
and the equipped norm
\begin{equation*}
    \|u\|_{\cW_p}^p = \|u\|_{L^p((0, T)\times \Omega\times \bbR^{d-1})}^p +  \|\partial_z u + \bk\cdot \nabla u\|_{L^p((0, T)\times \Omega\times \bbR^{d-1})}^p.
\end{equation*}
\begin{theorem}\label{THM: UNIQUE}
Let the initial condition $f(\bx)$ satisfy:
\begin{equation*}
   (i)\qquad \|f\|_{\infty}\le  \frac{1}{\nu_{d-1}} \inf_{(z,\bx)\in (0,T)\times \Omega}\frac{\Sigma_{a, k-1} }{k \Sigma_{a, k} } , \quad \forall 1\le k\le L
\end{equation*}
or 
\begin{equation*}
         (ii)\qquad
         \YZ{\begin{cases}\nu_{d-1}\Sigma_a'(\|f\|_{\infty}) \|f\|_{\infty} \leq 2 \Sigma_s \underline{\theta}\\\nu_{d-1}\underline{\theta}\leq 1
         \end{cases}}, \quad \forall (z,\bx)\in (0, T)\times \Omega,
\end{equation*}
where $\Sigma_a'$ is the Fr\'echet derivative of $\Sigma_a$, that is, 
\begin{equation*}
    \Sigma_a'(m) = \YZ{\sum_{l=1}^L \Sigma_{a, l}(z, \bx) l m^{l-1}}\,.
\end{equation*}
Then equation~\eqref{EQ: RTE 1} admits a unique non-negative solution in $\cW_{\infty}((0, T)\times \Omega\times \bbS^{d-1})$.
\end{theorem}
\begin{proof}
Let $F: m\mapsto \phi$ be the map defined through the relation $\aver{\phi} = F m$, where $\phi$ solves the transport equation
\begin{equation}\label{EQ: RTE H}
\begin{aligned}
    \frac{\partial \phi}{\partial z} + \bk\cdot \nabla\phi +  \Sigma_a(m)  \phi + \Sigma_s \left(\phi - \cK \phi\right) &= 0\quad &\text{ in }&X,\\
 \phi (z, \bx, \bk) &= 0\quad&\text{ on }&\Gamma_{-},\\
    \phi (0, \bx, \bk) &= f(\bx)\quad &\text{ on }& X_0\,.
\end{aligned}
\end{equation}
We denote by $\cM$ the set
\begin{equation*}
    \cM = \{m \in L^{\infty}((0, T)\times \Omega)\cap L^2((0, T)\times \Omega)\mid 0\le m \le \|f\|_{\infty} + \delta\}
\end{equation*}
with $\delta > 0$ being arbitrary. It is straightforward to check that $\cM$ is convex, bounded, and closed under the usual $L^2$ topology. For any $m\in \cM$, we have that $\Sigma_a(m)$ is non-negative. Therefore, by the maximum principle for the linear transport equation~\eqref{EQ: RTE H} (see for instance~\cite{DaLi-Book93-6}), $\aver{\phi}\in\cM$. This shows that $F: \cM\to \cM$. To show that $F$ is a continuous operator, we denote by $\phi_1$ and $\phi_2$ the solutions to~\eqref{EQ: RTE H} corresponding to $m_1\in\cM$ and $m_2\in \cM$ respectively. We then introduce $w = \phi_1- \phi_2$. It is then clear that $w$ solves the linear transport equation:
\begin{equation}\label{EQ: RTE H w}
\begin{aligned}
    \frac{\partial w}{\partial z} + \bk\cdot \nabla w +  \Sigma_a(m_1) w + \Sigma_s \left(w - \cK w\right) &= (\Sigma_a(m_2) - \Sigma_a(m_1) ) \phi_2\quad &\text{ in }&X,\\
w (z, \bx, \bk) &= 0\quad&\text{ on }&\Gamma_{-},\\
   w (0, \bx , \bk) &=0\quad &\text{ on }& X_0\,.
\end{aligned}
\end{equation}
By standard linear theory~\cite{DaLi-Book93-6}, this equation admit a unique $w\in \cW_2$ such that
\begin{equation}
    \begin{aligned}
\|w\|_{\cW_2((0,T)\times \Omega\times \bbS^{d-1})}&\le C \| (\Sigma_a(m_2) - \Sigma_a(m_1) ) \phi_2\|_{L^2((0,T)\times \Omega\times\bbS^{d-1})}\\& \le C \Sigma_a'(\|f\|_{\infty}+\delta) \|f\|_{\infty} \|m_2 - m_1\|_{L^2((0,T)\times \Omega)}
    \end{aligned}
\end{equation}
for some constant $C > 0$. Using the averaging lemma~\cite{GoLiPeSe-JFA88}, we obtain that there exists a constant $C' > 0$ that 
\begin{equation}
\begin{aligned}
    \|Fm_1 - Fm_2\|_{H^{1/2}((0, T)\times \Omega)} &= \|\aver{w}\|_{H^{1/2}((0, T)\times \Omega)} \le C' \|w\|_{\cW_2((0,T)\times \Omega\times \bbS^{d-1})} \\
    &\le CC' \Sigma_a'(\|f\|_{\infty}+\delta) \|f\|_{\infty} \|m_2 - m_1\|_{L^2((0,T)\times \Omega)}.
\end{aligned}
\end{equation}
Combining this with the Kondrachov embedding theorem, we have shown that $F:\cM \to \cM$ is a continuous compact operator. By the Schauder fixed point theorem, there exists a fixed point for $\cM$, and hence~\eqref{EQ: RTE 1} has a non-negative solution.

To prove the uniqueness of the solution, we use  Lemma~\ref{LEM: KEL}. We first observe that for any fixed point $m^{\ast}$ of $F$, it must satisfy the conditions: $m^{\ast} \le \|f\|_{\infty} $ and $m^{\ast} > 0$. This is due to the fact that $f$ is strictly positive. Hence there are no fixed points on the boundary $\partial\cM$. Next, we show that $F'(m)$ cannot have $1$ as its eigenvalue. Let $\phi$ be the solution to~\eqref{EQ: RTE H} with $m\in\cM$ and $ \delta m$ be a perturbation that $m + \delta m\in\cM$, we denote by $w$ the solution to the following equation:
\begin{equation}\label{EQ: RTE H w 1}
\begin{aligned}
    \frac{\partial w}{\partial z} + \bk\cdot \nabla w +  \Sigma_a(m) w + \Sigma_s \left(w - \cK w\right) &= -\Sigma_a'(m)  \delta m  \phi \quad &\text{ in }&X,\\
w (z, \bx, \bk) &= 0\quad&\text{ on }&\Gamma_{-},\\
   w (0, \bx , \bk) &=0\quad &\text{ on }& X_0\,.
\end{aligned}
\end{equation}
Then the Fr\'echet derivative $F'(m)[\delta m] = \aver{w}$. Suppose $F'(m)$ indeed has $1$ as its eigenvalue and $\aver{w}$ as the corresponding nonzero eigenfunction. Then $F'(\aver{w}) = \aver{w}$ and
\begin{equation}\label{EQ: RTE H w 2}
\begin{aligned}
         \frac{\partial w}{\partial z} + \bk\cdot \nabla w +  \Sigma_a(m) w + \Sigma_s \left(w - \cK w\right) 
         &= -\Sigma_a'(m)  \aver{w}  \phi \quad &\text{ in }&X,\\
w (z, \bx, \bk) &= 0\quad&\text{ on }&\Gamma_{-},\\
   w (0, \bx , \bk) &=0\quad &\text{ on }& X_0\,.
\end{aligned}
\end{equation}
Using the notations $\Sigma_t = \Sigma_a(m) + \Sigma_s$ and $R = \Sigma_s \cK w - \Sigma_a'(m) \aver{w} \phi$, we can write the solution to ~\eqref{EQ: RTE H w 2} as
\begin{equation*}
\begin{aligned}
   w(z, \bx, \bk)  &= \int_0^{z \wedge \tau_{-}(\bx, \bk) } \exp^{-\int_0^s \Sigma_t (z-l, \bx - l\bk)dl} R(z - s, \bx-s\bk, \bk) ds \\
    &= \int_0^{z \wedge \tau_{-}(\bx, \bk) } \Sigma_t(z-s, \bx - s\bk )e^{-\int_0^s \Sigma_t (z-l, \bx - l\bk)dl}\frac{R(z - s, \bx-s\bk, \bk)}{ \Sigma_t(z-s, \bx - s\bk)} ds\,,
\end{aligned}
\end{equation*}
which then gives the bound
\begin{equation}\label{EQ:Bound Sol}
\begin{aligned}
    |w(z, \bx, \bk)| &\le \int_0^{z \wedge \tau_{-}(\bx, \bk) } \Sigma_t(z-s, \bx - s\bk )e^{-\int_0^s \Sigma_t (z-l, \bx - l\bk)dl}\left|\frac{R(z - s, \bx-s\bk, \bk)}{ \Sigma_t(z-s, \bx - s\bk)}\right| ds \\
    &\le \left(1 -  e^{-\int_0^{z \wedge \tau_{-}(\bx, \bk)} \Sigma_t (z-l, \bx - l\bk)dl}  \right) \sup_{0\le s\le z \wedge \tau_{-}(\bx, \bk)  } \left|\frac{R(z - s, \bx-s\bk, \bk)}{ \Sigma_t(z-s, \bx - s\bk)}\right|  \\
    &\le  \gamma \sup_{(0, T)\times \Omega\times \bbS^{d-1} } \left|\frac{R(z, \bx, \bk)}{ \Sigma_t(z, \bx)}\right|  
\end{aligned}
\end{equation}
for some $0< \gamma < 1$. Here $a\wedge b = \min(a, b)$ and $\tau_{-}(\bx, \bk)$ is the distance from $\bx$ to $\Gamma_{-}$ in the direction of $-\bk$. The next step is to show that $ \sup_{(0, T)\times \Omega\times \bbS^{d-1} } \left|\frac{R(z, \bx, \bk)}{ \Sigma_t(z, \bx)}\right|  \le \|w\|_{\infty}$, which, when combined with the bound in~\eqref{EQ:Bound Sol}, leads to the bound $|w(z, \bx, \bk)| \le \gamma \|w\|_{\infty}$ (and hence $w = 0$). This contradicts the assumption that $\aver{w}$ is the eigenfunction corresponding to the eigenvalue $1$ of $F'$. We derive the bound under the two assumptions in the theorem.
\begin{enumerate}
    \item [(a)] When condition (i) is satisfied, we deduce from it that
    \[
    	\nu_{d-1}\Sigma_a'(m) \|f\|_{\infty} \le \Sigma_a(m)\,.
    \]
    Meanwhile, we also have that
    \begin{equation*}\label{EQ: R}
        |R(z, \bx, \bk) |\le \Sigma_s \|w\|_{\infty} + \nu_{d-1}\Sigma_a'(m) \|f\|_{\infty} \|w\|_{\infty},
    \end{equation*}
    Combining the above two bounds gives us
\begin{equation*}
    |R(z, \bx, \bk)| \le ( \Sigma_s + \Sigma_a(m) )\|w\|_{\infty} = \Sigma_t  \|w\|_{\infty}.
\end{equation*}
Therefore, we have $ \sup_{(0, T)\times \Omega\times \bbS^{d-1} } \left|\frac{R(z, \bx, \bk)}{ \Sigma_t(z, \bx)}\right|  \le \|w\|_{\infty}$.
\item [(b)] For the case when condition (ii) is satisfied, we first observe that
\begin{equation*}
    R(z, \bx, \bk) = \Sigma_s \left( \cK w - \underline{\theta} \aver{w} \right) + \left( \Sigma_s \underline{\theta} -  \Sigma_a'(m) \phi \right) \aver{w}
\end{equation*}
implies
\begin{equation}\label{EQ:R Bound}
   \left| \frac{R(z, \bx, \bk)}{\Sigma_t(z, \bx)} \right| \le \frac{\Sigma_s (1 - \nu_{d-1}\underline{\theta}) + \nu_{d-1}|\Sigma_s \underline{\theta} - \Sigma_a'(m)\phi|}{\Sigma_t(z, \bx)} \|w\|_{\infty}.
\end{equation}
When $\|f\|_{\infty}$ satisfies
\begin{equation}\label{EQ: COND}
    \|f\|_{\infty} \le \frac{1}{\nu_{d-1}}\frac{2 \Sigma_s \underline{\theta}}{\Sigma_a'(\|f\|_{\infty} + \delta)}, \quad \forall (z,\bx)\in (0, T)\times \Omega, 
\end{equation}
we obtain that $ \Sigma_s (1 - \nu_{d-1}\underline{\theta}) + \nu_{d-1}|\Sigma_s \underline{\theta} - \Sigma_a'(m)\phi| \le    \Sigma_s $ which, when combined with~\eqref{EQ:R Bound}, implies that
\begin{equation*}
    \left| \frac{R(z, \bx, \bk)}{\Sigma_t(z, \bx)} \right| \le \|w\|_{\infty}.
\end{equation*}
Since $\delta > 0$ is arbitrary, taking $\delta \to 0$, the condition~\eqref{EQ: COND} becomes
\begin{equation*}
       \|f\|_{\infty} \leq \frac{1}{\nu_{d-1}}\frac{2 \Sigma_s \underline{\theta}}{\Sigma_a'(\|f\|_{\infty})}, \quad \forall (z,\bx)\in (0, T)\times \Omega\,,
\end{equation*}
which is simply (ii).
\end{enumerate}
The proof is now complete.
\end{proof}

To study the diffusion limit, we introduce a small parameter $\epsilon$ and denote by $W_{\epsilon}$ the solution to the scaled radiative transport equation:
\begin{equation}\label{EQ:RTE}
\begin{aligned}
\frac{\partial W_{\epsilon}}{\partial z} + \frac{1}{\epsilon} \bk\cdot \nabla W_{\epsilon} +  \Sigma_a(\aver{W_{\epsilon}})  W_{\epsilon} + \frac{1}{\epsilon^2}\Sigma_s \left(W_{\epsilon} - \cK W_{\epsilon}\right) &= 0\quad &\text{ in }&X,\\
  W_{\epsilon}(z, \bx, \bk) &= 0\quad&\text{ on }&\Gamma_{-},\\
    W_{\epsilon}(0,\bx,  \bk) &= f(\bx)\quad &\text{ on }& X_0\,.
\end{aligned}
\end{equation}
We have the following corollary using condition (ii) in Theorem~\ref{THM: UNIQUE}.
\begin{corollary}
If $\epsilon$ is sufficiently small such that 
\begin{equation*}
    \Sigma_a'(\|f\|_{\infty})\|f\|_{\infty} \leq \frac{1}{\nu_{d-1}}\frac{2  \underline{\theta} \Sigma_s  }{\epsilon^2},
\end{equation*}
then ~\eqref{EQ:RTE} admits a unique non-negative solution in $\cW_{\infty}((0, T)\times \Omega\times \bbS^{d-1})$. Moreover, the solution satisfies $\|W_{\epsilon}\|_{\infty} \le \|f\|_{\infty}$.
\end{corollary}

\subsection{Asymptotic expansion}

Let $\epsilon$ be sufficiently small such that~\eqref{EQ:RTE} admits a unique solution. We formally expand the solution in powers of $\epsilon$: 
\begin{equation}
    W_{\epsilon}(z, \bx, \bv) = W_0 (z, \bx, \bv) + \epsilon W_1(z, \bx, \bv) + \epsilon^2 W_2(z, \bx, \bv) + \phi_{\epsilon}(z, \bx, \bv)\,.
\end{equation}
\YZ{Let $\cI$ denote the identity operator. We then substitute the above expansion into~\eqref{EQ:RTE}. Matching the equations at orders $\epsilon^{-2}$, $\epsilon^{-1}$ and $\epsilon^0$ gives the system:}
\begin{equation}\label{EQ: DIFF}
\begin{aligned}
     (\cI - \cK) W_0 &= 0, \\
     \Sigma_s(\cI - \cK) W_1 + \bk\cdot \nabla W_0 &= 0,\\
     \frac{\partial W_0}{\partial z} + \Sigma_s(\cI - \cK) W_2 + \bk\cdot \nabla W_1 + \Sigma_a ({W_0}) W_0 &= 0\,.
\end{aligned}
\end{equation}
Following standard procedures~\cite{CaSc-Book21, DaLi-Book93-6}, we obtain from the first two equations that
\begin{equation}
    W_0(z, \bx, \bk) = W_0(z, \bx), \quad W_1(z, \bx, \bk) = -\sum_{i=1}^d\frac{D_i(\bk)}{\Sigma_s} \partial_{x_i} W_0(z, \bx)\,,
\end{equation}
where $D_i(v)\in L^{\infty}(\bbS^{d-1})$ are the unique solutions to 
\begin{equation}
    (\cI - \cK) D_i(\bk) = \bk\cdot \be_i,\quad \int_{\bbS^{d-1}}D_i(\bk) d\bk = 0\,,\quad i = 1,2, \dots, d.
\end{equation}
Next, we integrate the third equation in~\eqref{EQ: DIFF} over $\bbS^{d-1}$, and utilize the fact that 
\[
    \aver{(\cI-\cK)W_2} = 0
\]
to get the equation for $W_0$. This leads to 
\begin{equation*}
   \int_{\bbS^{d-1}} \left( \frac{\partial W_0}{\partial z} + \left( \bk\cdot \nabla W_1 + \Sigma_a(W_0) W_0 \right) \right)   d\bk = 0\,.
\end{equation*}
Since $W_0$ is independent of $\bk$, we have that $W_0$ solves the diffusion equation:
\begin{equation}\label{EQ: DIFFUSION}
    \begin{aligned}  \frac{\partial W_0}{\partial z} - \nabla \cdot \left( \frac{A}{\Sigma_s} \nabla W_0 \right)  + \Sigma_a (W_0) W_0 &= 0, \quad &\text{in }& (0,T)\times \Omega, \\
     W_0(z, \bx) &= 0,\quad&\text{ on }&(0,T)\times \partial \Omega,\\
    W_0(0, \bx) &= f(\bx)\quad &\text{ on }&  \Omega,
    \end{aligned}
\end{equation}
where the matrix $A$ is positive definite and is defined as
\begin{equation*}
    A_{ij} = \int_{\bbS^{d-1}} \bk \cdot \be_i D_j(\bk) d\bk\,.
\end{equation*}
Under our assumption that the scattering phase function $p(\bk\cdot \bk')$ is rotation invariant, we have that $A_{ij} = \frac{1}{d(1-g)}\delta_{ij}$, where $g\in(-1,1)$ is the anisotropy parameter. 

Let $\phi_{\epsilon}$ be the remainder in the expansion. Then we can check that $\phi_{\epsilon}$ satisfies the equation:
\begin{equation}\label{EQ: PHI}
\begin{aligned}
    \partial_z \phi_{\epsilon} + \frac{1}{\epsilon}\bk \cdot \nabla \phi_{\epsilon} +  \frac{\Sigma_s}{\epsilon^2}\left(\cI - \cK\right)\phi_{\epsilon} +\Sigma_a(\aver{W_{\epsilon}} )\phi_{\epsilon} &= \epsilon h_1,\; &\text{ in }& X\\
  \phi_{\epsilon}(z, \bx, \bk) &= \epsilon h_2,\quad&\text{ on }&\Gamma_{-},\\
    \phi_{\epsilon}(0,\bx,  \bk) &= \epsilon h_3, \quad &\text{ on }& X_0\,,
\end{aligned}
\end{equation}
where the functions $h_1$, $h_2$,  and $h_3$ are respectively of the forms:
\begin{equation*}
\begin{aligned}
h_1(z, \bx, \bk) &= -\partial_z W_1 -\bk\cdot \nabla W_2 - \Sigma_a(\aver{W_{\epsilon}}) ( W_1 + \epsilon W_2) \\&\quad - \epsilon \partial_{z} W_2 + \frac{1}{\epsilon}\left[\Sigma_a(W_0) - \Sigma_a(\aver{W_{\epsilon}})\right] W_0 ,\\
h_2(z, \bx,\bk) & = -\sum_{i=1}^d\frac{D_i(\bk)}{\Sigma_s} \frac{\partial W_0}{\partial x_i} - \epsilon W_2, \\
h_3(0, \bx, \bk) &=  -W_1 - \epsilon W_2\,.
\end{aligned}
\end{equation*}
Let $\delta := \frac{1}{2}\inf_{[0, T]\times \Omega} \Sigma_{a, 0}(z, \bx)$. Then by the assumtion in~\eqref{EQ:Absorption Coeff}, $\delta>0$. We then have that for any $z\in [0, T)$,
\begin{equation}\label{EQ: PHI EST}
\begin{aligned}
    \|\phi_{\epsilon} (z, \cdot, \cdot) \|_{\infty} &\le \epsilon \|h_3\|_{\infty} e^{-\delta z} + \epsilon \int_0^{z} e^{-\delta (z- s)}(\|h_1\|_{\infty} + \delta \|h_2\|_{\infty}) ds \\
    &\le \epsilon  T (\|h_1\|_{\infty} + \delta \|h_2\|_{\infty})   + \epsilon\|h_3\|_{\infty}.
\end{aligned}
\end{equation}
It remains to show that $h_1$, $h_2$, and $h_3$ are bounded.  We first observe from the equations in~\eqref{EQ: DIFF} that 
\begin{equation*}
    \|W_1\|_{\infty} \le C_1 \| W_0\|_{C([0,T), C^1(\Omega))}, \quad \|W_2\|_{\infty} \le C_2 \|W_0\|_{C([0,T), C^2(\Omega))}
\end{equation*}
We then take the derivative with respect to $z$ of the equations in~\eqref{EQ: DIFF} to deduce that
\begin{equation*}
 \|\partial_z W_1\|_{\infty} \le C_3 \|W_0\|_{C([0,T), C^3(\Omega))},\quad \|\partial_z W_2\|_{\infty} \le C_4 \|W_0\|_{C([0,T), C^4(\Omega))}    
\end{equation*}
together with 
\[
\|\bk\cdot \nabla W_2\|_{\infty} \le C_5 \|W_0\|_{C([0,T), C^3(\Omega))}\,.
\]
Therefore, given that $W_0\in C([0, T), C^4(\Omega))$, we have the following estimate for the diffusion approximation:
\begin{equation*}
\|W_{\epsilon} - W_0\|_{\infty}\le \|\phi_{\epsilon}\|_{\infty}+  \epsilon \|W_1\|_{\infty} + \epsilon^2 \|W_2\|_{\infty} = C(T, \|W_0\|_{C([0, T), C^4(\Omega)) }) \epsilon.
\end{equation*}
To ensure the regularity of the solution $W_0$, at least for a short time, we simply need the initial condition $f$ to be smooth enough since $W_0$ would be smoother than the initial condition due to the diffusive nature.

We can prove the following result.
\begin{theorem}\label{THM: REG}
	Assume that $f\in C^{4}_0(\Omega)$ is non-negative, and the absorption and scattering coefficients satisfy the condition
	\[
   		\Sigma_{a,l}\in C^{2}(\Omega),\quad \Sigma_{a, l}\ge 0, \quad \Sigma_s > 0.
	\]
	Then the diffusion equation~\eqref{EQ: DIFFUSION} admits a unique strong solution $W_0\in C([0, T), C^{4}(\Omega))$ when $1\le d\le 3$. 
\end{theorem}
\begin{proof}
Let $L = -\nabla\cdot \left(\frac{A}{\Sigma_s}\nabla\right)$. Then $-L$ is the infinitesimal generator of an analytic semigroup $G(t)$ on $L^2(\Omega)$ and $\|G(t)\|\le 1$ for all $t\ge 0$. We denote 
\begin{equation*}
    \cD(L) =  H^2(\Omega)\cap H^1_0(\Omega).
\end{equation*}
By~\cite[Theorem 8.4.4]{Pazy-Book12} and~\cite[Theorem 6.3.1]{Pazy-Book12}, we have that when $f\in \cD(L)$, there exists a unique local strong solution $W_0$ to the equation~\eqref{EQ: DIFFUSION} on $(0, T)\times \Omega$, that is,
\begin{equation}
    W_0\in C([0, T), L^2(\Omega))\cap C((0, T), H^2(\Omega)\cap H^1_0(\Omega)) \cap C^1((0, T), L^2(\Omega))\,.
\end{equation}
Moreover, $0\le W_0 \le \|f\|_{\infty}$ by the comparison principle. Then the result of ~\cite[Corollary 6.3.2]{Pazy-Book12} ensures that $\partial_z W_0$ is locally H\"older continuous for $z\in(0, T)$. Hence $W_0(z, \bx)$ and $\partial_z W_0(z, \bx)$ are both continuous on $(0,T)\times \overline{\Omega}$. This means that $\Sigma_a(W_0)$ is also continuous. Therefore we must have $W_0(z, \cdot)\in C^{2}(\Omega)$, which means $W_0$ is a classical solution. 

Let $g(\bx) :=  -Lf - \Sigma_a(f)$, $f \in C^2(\Omega)\in \cD(L)$. By differentiating~\eqref{EQ: DIFFUSION}, we find that $\psi := \partial_z W_0$ satisfies the following equation:
\begin{equation}
        \begin{aligned}
      \partial_z \psi + L \psi + \left( \Sigma_a'(W_0) W_0 +  \Sigma_a(W_0) \right) \psi &= 0\quad &\text{ in }&(0,T)\times \Omega,\\
  \psi(z, \bx) &= 0\quad&\text{ on }&(0, T)\times\partial \Omega,\\
    \psi(0,\bx) &= g(\bx)\quad &\text{ on }& \{0\}\times \Omega.
\end{aligned}   
\end{equation}
Following a similar process, we can deduce that $\psi(z, \cdot) \in C^{2}(\Omega)$. Since $W_0(z,\cdot)\in C^2(\Omega)$ and  $\partial_z W_0(z,\cdot)\in C^2(\Omega)$ for $z\in (0, T)$, we have $\partial_z W_0(z,\cdot) + \Sigma_a(W_0(z,\cdot)) W_0(z,\cdot)\in C^2(\Omega)$. By classical regularity theory for elliptic equations~\cite{GiTr-Book00}, $W_0(z,\cdot)\in C^4(\Omega)$.
\end{proof}


\begin{remark}
We have assumed so far that the initial condition $f$ is independent of the variable $\bk$. In fact, the case of $f$ depending on $\bk$, that is, $f = f(\bx, \bk)$, can be treated in a similar manner by introducing another fast variable $\theta = \frac{z}{\epsilon^2}$, as in \cite[Section~XXI.5.3]{DaLi-Book93-6}. We will not repeat the calculations here.
\end{remark}

\subsection{The case of degenerate coefficients}

Let us now briefly consider the case when the problem is degenerate, that is, when the absorption coefficient can vanish in part of the domain of interest. More precisely, we relax the requirement that all $\Sigma_{a,l} > 0$ to the following:
\begin{equation}
   \Sigma_{a,l} \ge 0, \quad  \Sigma_a (\|f\|_{\infty}) > 0, \quad \forall (z, \bx)\in [0, T]\times \Omega.
\end{equation}
In this case, $\Sigma_a'(\|f\|_{\infty}) > 0$. When $\epsilon$ is sufficiently small, the scaled transport equation~\eqref{EQ:RTE} admits a unique solution in $L^{\infty}(X)$. Let $w_\epsilon$ be the solution to the following linear transport equation:
\begin{equation}\label{EQ: W}
    \begin{aligned}
    \partial_z w_{\epsilon} + \frac{1}{\epsilon}\bk\cdot \nabla w_{\epsilon} + \Sigma_a (\|f\|_{\infty}) w_{\epsilon} + \frac{\Sigma_s}{\epsilon^2 }(\cI-\cK)w_{\epsilon} &= 0\quad &\text{ in }&X,\\
  w_{\epsilon}(z, \bx, \bk) &= 0\quad&\text{ on }&\Gamma_{-},\\
    w_{\epsilon}(0,\bx,  \bk) &= f(\bx)\quad &\text{ on }& X_0.
\end{aligned}
\end{equation}
Since the absorption coefficient $\Sigma_a (\|f\|_{\infty}) \ge \Sigma_a(\aver{W_{\epsilon}})$, we conclude that $w_{\epsilon} \le W_{\epsilon}$ for any $\epsilon > 0$. On the other hand, as $\epsilon \to 0$, we have $w_{\epsilon} \to w_0$, where $w_0$ is the solution to
\begin{equation}\label{EQ: W 1}
    \begin{aligned}
    \partial_z w_{0} - \nabla\cdot \left(\frac{A}{\Sigma_s} \nabla w_0\right) +\Sigma_a (\|f\|_{\infty}) w_0 &= 0\quad &\text{ in }&(0, T)\times \Omega,\\
  w_{0}(z, \bx) &= 0\quad&\text{ on }&(0, T)\times \partial \Omega,\\
    w_{0}(0,\bx) &= f(\bx)\quad &\text{ on }& \Omega.
\end{aligned}
\end{equation}
Hence $\|W_{\epsilon}\|\ge \|w_{\epsilon}\|\ge \inf_{[0, T]\times \Omega} w_0 - C\epsilon$ for some $C > 0$. Because $\inf_{(0, T)\times \Omega} w_0 
> 0$ strictly for $\epsilon$ sufficiently small, we conclude that $W_{\epsilon}$ is bounded from below by a positive number, which implies $\Sigma_a(\aver{W_{\epsilon}})$ is also strictly positive. Then repeat the process in~\eqref{EQ: PHI EST} by setting $\delta = \frac{1}{2}\inf_{(0, T)\times \Omega} \Sigma_a(\aver{W_{\epsilon}})$ instead, we obtain the same conclusion that $\|W_{\epsilon} - W_0\|_{\infty} \le C'\epsilon$ for a constant $C'$ independent of $\epsilon$.

\section{An application in inverse problems}
\label{SEC:IP}

We now consider the following inverse medium problem as a direct application of the transport model:
\begin{equation}\label{EQ:FLUO}
        \begin{aligned}
       \partial_z u + \bk\cdot \nabla u + \Sigma_a(\aver{u}) u + \Sigma_s (\cI - \cK) u &= 0\quad &\text{ in }&X,\\
  u(z, \bx, \bk) &= 0\quad&\text{ on }&\Gamma_{-},\\
   u(0,\bx,  \bk) &= f(\bx)\quad &\text{ on }& X_0,
\end{aligned}
\end{equation}
where $f(\bx)\in L^{\infty}(X_0)$ is a strictly positive source function. We assume that~\eqref{EQ:FLUO} has a unique positive solution $u\in \cW_{\infty}$. 

We assume that the absorption coefficient $\Sigma_a$ is not known, but we have additional data that is the density of the solution, that is, 
\begin{equation}\label{EQ:Data}
	g(z, \bx)=\aver{u}:=\int_{\bbS^{d-1}} u(z, \bx, \bk) d\bk\,.
\end{equation}
The inverse problem amounts to finding the unknown absorption coefficients $\Sigma_{a, l}$ from the observed data $g$ from a given $f$.

We can prove the following result.
\begin{theorem}\label{LEM: UNIQ}
Let $g$ and $\widetilde g$ be data of the form~\eqref{EQ:Data} generated from~\eqref{EQ:FLUO} with coefficients $\Sigma_{a}$ and $\widetilde \Sigma_a$ respectively. Then $g=\widetilde g$ implies $\Sigma_a(\aver{u})=\widetilde\Sigma_a(\aver{\widetilde u})$.
\end{theorem}
\begin{proof}
Let $\delta u = u - \widetilde{u} $. We verify that for any $\delta u$, we have the identity
\begin{equation}\label{EQ:Identity1}
    \int_{\Omega\times \bbS^{d-1}} \left[\bk\cdot \nabla \delta u\right] \frac{\delta u}{\tilde u} d\bx d\bk = \int_{\Omega\times \bbS^{d-1}} \bk\cdot \nabla \frac{|\delta u|^2}{2\tilde u} d\bx d\bk - \int_{\Omega\times \bbS^{d-1}} \left[\bk\cdot \nabla \frac{1}{\tilde u}\right] \frac{|\delta u|^2}{2} d\bx d\bk\,,
\end{equation}
and the identity
\begin{equation}\label{EQ:Identity2}
\bk\cdot \nabla \frac{1}{\tilde u} = \frac{1}{\tilde u^2}\partial_z \tilde u + \frac{\widetilde{\Sigma_a}(\aver{\tilde u})}{\tilde u} + \frac{\Sigma_s(\cI - \cK)\tilde u }{|\tilde u|^2}\,.
\end{equation}
Using the fact that $g=\wt g$, we can also conclude that
\begin{equation}\label{EQ:Identity3}
\aver{\delta u}=0\,.
\end{equation}

It is also straightforward to check that $\delta u$ solves the transport equation:
\begin{equation}
     \partial_z \delta u + \bk\cdot \nabla \delta u + {\Sigma_a}(\aver{{u}}) \delta u + \Sigma_s (\cI - \cK) \delta u= (\widetilde{\Sigma}_a(\aver{\Tilde{u}}) -\Sigma_a(\aver{{u}}) ) \tilde u\,,
\end{equation}
with zero initial and incoming boundary conditions. We multiply this equation by $\frac{\delta u}{\tilde u}$ and integrate over $\Omega\times \bbS^{d-1}$ to obtain
\begin{equation}\label{EQ:WEAK}
\begin{aligned}
     &\int_{\Omega\times \bbS^{d-1}} (\partial_z \delta u )\frac{\delta u}{\tilde u} d\bx d\bk + \int_{\Omega\times \bbS^{d-1}} \bk\cdot \nabla \frac{|\delta u|^2}{2\tilde u} d\bx d\bk - \int_{\Omega\times \bbS^{d-1}} \frac{|\delta u|^2}{2|\tilde u|^2} \partial_z \tilde u d\bx d\bk \\
     &- \int_{\Omega\times \bbS^{d-1}} \frac{\widetilde{\Sigma_a}(\aver{\tilde u}) + \Sigma_s}{2\tilde u} |\delta u|^2 d\bx d\bk + \int_{\Omega\times \bbS^{d-1}} \frac{(\Sigma_s\cK \tilde u)|\delta u|^2}{\YZ{2}|\tilde u|^2} d\bx d\bk \\& + \int_{\Omega\times \bbS^{d-1}} \frac{\Sigma_a(\aver{u})}{\tilde u} |\delta u|^2 d\bx d\bk+ \Sigma_s \int_{\Omega\times \bbS^{d-1}} \left[(\cI - \cK) \delta u\right] \frac{\delta u}{\tilde u}d\bx d\bk \\ &= \int_{\Omega\times \bbS^{d-1}} (\widetilde{\Sigma_a}(\aver{\Tilde{u}}) -\Sigma_a(\aver{{u}}) ) \delta u d\bx d\bk.
\end{aligned}
\end{equation}
where we have used the identities~\eqref{EQ:Identity1} and ~\eqref{EQ:Identity2}. 

We first observe that since $\Sigma_a(\aver{{u}})$ and $\widetilde{\Sigma_a}(\aver{\Tilde{u}})$ do not depend on $\bk$, the right hand side of~\eqref{EQ:WEAK} vanishes due to~\eqref{EQ:Identity3}. 

To handle the left hand side of~\eqref{EQ:WEAK}, we observe that
\begin{equation}\label{EQ:EQ1}
           \int_{\Omega\times \bbS^{d-1}} \bk\cdot \nabla \frac{|\delta u|^2}{2\tilde u} d\bx d\bk = \int_{\Gamma_{+}} \bk\cdot \bn \frac{|\delta u|^2}{2\tilde u} \ge 0\,,
\end{equation}
\begin{equation}\label{EQ:EQ2}
           (\partial_z \delta u)\frac{\delta u}{\tilde u}= \frac{1}{2}\partial_z \left[\frac{|\delta u|^2}{\tilde u}\right] + \frac{1}{2} \frac{|\delta u|^2}{|\tilde u|^2} \partial_z \tilde u\,,
\end{equation}
and 
\begin{equation}\label{EQ:EQ3}
       \Sigma_s \int_{\Omega\times \bbS^{d-1}} \left[(\cI - \cK)\delta u \right] \frac{\delta u}{\tilde u} d\bx d\bk = \Sigma_s \int_{\Omega\times \bbS^{d-1}} \frac{|\delta u|^2}{\tilde u} d\bx d\bk - \Sigma_s \int_{\Omega\times \bbS^{d-1}} (\cK \delta u)\frac{\delta u}{\tilde u} d\bx d\bk\,.
\end{equation}
We can also prove the following inequality (see \Cref{SEC:Proof of Inequality1}), 
\begin{equation}\label{EQ:Inequality1}
      \int_{\Omega\times \bbS^{d-1}} ( \cK \delta u ) \frac{\delta u}{\tilde u} d\bx d\bk \le\int_{\Omega\times \bbS^{d-1}} (\cK \tilde u) \frac{|\delta u|^2}{\YZ{2}|\tilde u|^2} d\bx d\bk + \nu_{d-1}\frac{\kappa^2}{\YZ{2}} \int_{\Omega\times \bbS^{d-1}} \frac{|\delta u|^2}{\tilde u} d\bx d\bk,
\end{equation}
where $\kappa = \frac{(\overline{\theta} - \underline{\theta})}{2\sqrt{\underline\theta}}$, the constants $\overline{\theta}$ and $\underline{\theta}$ are defined in~\eqref{EQ: THETA}. 

Let $M(z, \bx) :=  \Sigma_a(\aver{u}) - \frac{\widetilde{\Sigma_a}(\aver{\tilde u})}{2} +\Sigma_s \left(  \frac{\YZ{1} -\nu_{d-1} \kappa^2}{\YZ{2}}\right) $. We can then deduce from~\eqref{EQ:WEAK}, using~\eqref{EQ:EQ1},~\eqref{EQ:EQ2},~\eqref{EQ:EQ3}, and~\eqref{EQ:Inequality1}, that
\begin{equation}
\begin{aligned}
    \frac{1}{2} \partial_z \int_{\Omega\times \bbS^{d-1}}  \frac{|\delta u|^2}{\tilde u} d\bx d\bk + \int_{\Omega\times \bbS^{d-1}} M(z, \bx)\frac{|\delta u|^2 }{\tilde u} d\bx d\bk \le  0.      
\end{aligned}
\end{equation}
Since the coefficients $\Sigma_{a,l}$ and $\Sigma_s$ are finite and both $\aver{u}, \aver{\tilde {u}}$ are bounded by $\|f\|_{L^{\infty}(X_0)}$,  there exists a constant $\underline{M} = \inf_{(0, T)\times \Omega} M(z, \bx)$ that $$  \frac{1}{2} \partial_z \int_{\Omega\times \bbS^{d-1}}  \frac{|\delta u|^2}{\tilde u} d\bx d\bk + \underline{M} \int_{\Omega\times \bbS^{d-1}} \frac{|\delta u|^2 }{\tilde u} d\bx d\bk\le 0 .$$ Then by the Gr\"onwall inequality and the initial condition $\delta u = 0$ at $z = 0$, we must have that $\delta u \equiv 0$. Therefore $u = \tilde u$, which implies that 
\begin{equation}
    \Sigma_a(\aver{u}) = -\frac{\partial_z  u + \bk \cdot \nabla u + \Sigma_s(\cI -\cK)u}{u} = -\frac{\partial_z  \tilde u + \bk \cdot \nabla \tilde u + \Sigma_s(\cI -\cK)\tilde u}{\tilde u} = \widetilde{\Sigma_a}(\aver{\tilde u})\,.
\end{equation}
The proof is complete.
\end{proof}
The following corollary is a direct result of the comparison principle and Theorem~\ref{LEM: UNIQ}.
\begin{corollary}
Under the assumption of Theorem~\ref{LEM: UNIQ}, the coefficients $\Sigma_{a,l}$ can be uniquely determined with finitely many data sets $\aver{u_j}$, $j=1,2,\dots, L+1$ if the initial conditions satisfy $0 < f_1 < f_2 < \dots < f_{L+1}$ on $\Omega$.
\end{corollary}
\section{Concluding remarks}
\label{SEC:Concl}

This work describes the derivation of semilinear radiative transport models for wave propagation in highly-scattering media with nonlinear absorption. While the technical aspects of the derivation are relatively standard, we believe that our work provides a theoretical justification for the semilinear radiative transport models, as well as their diffusion approximations, used in applications such as multi-photon imaging~\cite{BaReZh-JBO18,LaReZh-SIAM22,ReZh-SIAM18,ReZh-SIAM21,StZh-SIAM22,YaNaTa-OE11,YeKo-OE10,YiLiAl-AO99}. 

As we have remarked before, one concrete example of the quadratic absorption we considered here is two-photon absorption in nonlinear optics~\cite{Mahr-QE12,RuPe-AOP10}. The radiative transport equation we derive for this case, equation~\eqref{EQ: TRANS}, is different from the two-photon radiative transport equation of~\cite{MaSc-PRE14} where the phase space intensity corresponds to a two-photon entangled state of light, not two-photon absorption.

The calculation we have presented here does not generalize to media with nonlinearities such as those that arise in the Kerr effect and second harmonic generation~\cite{AsZh-JST20,GoPs-PRA11,ReLi-OE10,SzKi-JOSA18}. The derivation of transport equations for wave propagation in such media is a topic of great interest, but is much more challenging due to the richness of the  behavior of the corresponding wave equations~\cite{CoAnRu-PRA07,CoLe-PRB11,Skipetrov-PRE03}. We point to the derivation in~\cite{FaJiPa-SIAM03} in the context of the nonlinear Schr\"{o}dinger equation within the mean-field approximation, and leave further investigations in this direction to future work. We note that the acousto-optic effect, in which light undergoes a frequency shift due to interaction with an acoustic wave, has also been studied in random media. Although this effect is not nonlinear in the sense considered in this paper, it is possible to develop a suitable kinetic model and associated radiative transport equations~\cite{HoSc-PRE17}.

\section*{Acknowledgment}

The work of JK is supported in part by the Simons Foundation Math + X Investigator Award \#376319 (to Michael I. Weinstein). KR is partially supported by the National Science Foundation through grants DMS-1913309 and DMS-1937254. JCS acknowledges support from the NSF Grant DMS-1912821 and the AFOSR Grant FA9550-19-1-0320.

\appendix
\section{Proof of inequality\texorpdfstring{~\eqref{EQ:Inequality1}}{}}
\label{SEC:Proof of Inequality1}
\begin{proof}
Using the AM-GM inequality, we deduce that
\begin{equation*}
    \int_{\Omega\times \bbS^{d-1}} ( \cK \delta u ) \frac{\delta u}{\tilde u} d\bx d\bk \le \int_{\Omega\times \bbS^{d-1}} (\cK \tilde u) \frac{|\delta u|^2}{\YZ{2}|\tilde u|^2} d\bx d\bk + \frac{1}{\YZ{2}}\int_{\Omega\times \bbS^{d-1}} \left|\frac{ \cK \delta u }{\sqrt{\cK \tilde u}}\right|^2 d\bx d\bk.
\end{equation*}
Since $\underline{\theta}\le p(\bk, \bk')\le \overline{\theta}$ and $\aver{\delta u} = 0$, we have the preliminary estimate
\begin{equation*}
 \left|  \frac{ \cK \delta u }{\sqrt{\cK \tilde u}} \right|= \left| \frac{\cK \delta u - \frac{\overline{\theta} + \underline{\theta}}{2}\aver{\delta u}  }{\sqrt{\cK \tilde u}}\right| \le \frac{(\overline{\theta} - \underline{\theta})\aver{|\delta u|}}{2\sqrt{ \underline{\theta}\aver{\tilde u} }}.
\end{equation*}
Denote the constant $\kappa := \frac{(\overline{\theta} - \underline{\theta})}{2\sqrt{\underline\theta}}$. We then have
\begin{equation*}
    \int_{\Omega\times \bbS^{d-1}} \left|\frac{ \cK \delta u }{\sqrt{\cK \tilde u}}\right|^2 d\bx d\bk \le  \kappa^2 \int_{\Omega\times \bbS^{d-1}} \frac{\aver{|\delta u|}^2}{{ \aver{\tilde u} }} d\bx d\bk =  \nu_{d-1} \kappa^2\int_{\Omega} \frac{\aver{|\delta u|}^2}{{ \aver{\tilde u} }} d\bx.
\end{equation*}
Using the Cauchy-Schwartz inequality, we arrive at
\begin{equation*}
 \int_{\bbS^{d-1}}\frac{|\delta u(\bx, \bk)|^2}{\tilde u(\bx, \bk)} d\bk \int_{\bbS^{d-1}} \tilde u(\bx, \bk) d\bk \ge \left( \int_{\bbS^{d-1}} |\delta u(\bx, \bk)| d\bk \right)^2 = \aver{|\delta u|} ^2\,.
\end{equation*}
Therefore, we have
\begin{equation*}
    \frac{\aver{|\delta u|}^2}{{ \aver{\tilde u} }} \le \int_{\bbS^{d-1}}\frac{|\delta u(\bx, \bk)|^2}{\tilde u(\bx, \bk)} d\bk.
\end{equation*}
This implies that 
\begin{equation*}
     \int_{\Omega\times \bbS^{d-1}} ( \cK \delta u ) \frac{\delta u}{\tilde u} d\bx d\bk \le\int_{\Omega\times \bbS^{d-1}} (\cK \tilde u) \frac{|\delta u|^2}{\YZ{2}|\tilde u|^2} d\bx d\bk + \nu_{d-1}\frac{ \kappa^2}{\YZ{2}} \int_{\Omega\times \bbS^{d-1}} \frac{|\delta u|^2}{\tilde u} d\bx d\bk.
\end{equation*}
This completes the proof.
\end{proof}

\bibliographystyle{siam}
\bibliography{RH-BIB}

\end{document}